\documentclass[11pt,reqno]{amsart}
\usepackage{amsmath,amsthm,amssymb,enumerate}
\usepackage{hyperref}
\usepackage{mathrsfs}
\usepackage{nicefrac}
\usepackage{mhequ}
\usepackage{appendix}
\usepackage{stmaryrd}
\usepackage{stmaryrd}
\usepackage{mathtools}
\DeclarePairedDelimiter{\ceil}{\lceil}{\rceil}

\usepackage{color}

\usepackage[a4paper,innermargin=1.25in,outermargin=1.25in,bottom=1.5in,marginparwidth=1.5in,marginparsep=3mm]{geometry}

\numberwithin{equation}{section}

\newtheorem{theorem}{Theorem}[section]
\newtheorem{lemma}[theorem]{Lemma}

\theoremstyle{definition}

\theoremstyle{remark}
\newtheorem{remark}{Remark}[section]

\newcommand\bR{\mathbb{R}}
\newcommand\bP{\mathbb{P}}

\newcommand\cD{\mathcal{D}}

\newcommand\cF{\mathcal{F}}

\newcommand{\bN}{\mathbb{N}}

\newcommand{\E}{\mathbb{E}}

\newcommand{\R}{\mathbb{R}}
\newcommand{\bit}{\begin{itemize}}
\newcommand{\eit}{\end{itemize}}
\newcommand\D{\partial}

\newcommand{\one}{\mathbf{1}}

\newcommand{\eps}{\varepsilon}

\newcommand{\N}{\mathbb{N}}

\begin{document}

\author{K. Dareiotis and M. Gerencs\'er}

\address[K. Dareiotis]{Max Planck Institute for Mathematics in the Sciences, Leipzig / University of Leeds}
\email{k.dareiotis@leeds.ac.uk}

\address[M. Gerencs\'er]{Institute of Science and Technology Austria/ Vienna University of Technology }
\email{mate.gerencser@tuwien.ac.at}

\title[Euler-Maruyama scheme with irregular drift]{On the regularisation of the noise for the Euler-Maruyama scheme with irregular drift}

\date{\today}

\begin{abstract}
The strong rate of convergence  of the Euler-Maruyama scheme for nondegenerate SDEs with irregular drift coefficients is considered.
In the case of $\alpha$-H\"older drift in the recent literature the rate $\alpha/2$ was proved in many related situations.
By exploiting the regularising effect of the noise more efficiently, we show that the rate is in fact arbitrarily close to $1/2$ for all $\alpha>0$. 
The result extends to Dini continuous coefficients, while in $d=1$ also to all bounded measurable coefficients.
\end{abstract}
\maketitle

\section{Introduction and main results}
We consider stochastic differential equations (SDEs)
\begin{equ}\label{eq:main}
dX_t=b(X_t)\,dt+\,dW_t,\quad X_0=x_0,
\end{equ}
on a fixed time horizon $[0,T]$, driven by a $\bR^d$-valued Wiener process $W=(W^1,...,W^d)$
on a filtered probability space $(\Omega, \mathcal{F}, (\mathcal{F}_t)_{t\in[0,T]}, \bP)$ with $(\mathcal{F})_{t \in [0,T]}$ satisfying the usual conditions, 
where $x_0$ is an $\R^d$-valued $\cF_0$-measurable random variable with finite variance, and the drift coefficient $b=(b^1,\ldots,b^d)$ is a measurable function on $\R^d$
with values in $\R^d$.
It is well known that equation \eqref{eq:main} has the remarkable property that even if $b$ is only known to be bounded and measurable
then a unique strong solution exists \cite{Zvonkin, Veretennikov}, for further developments see among others \cite{GyMart, KR, Davie, FF, Shap}.

It has recently been of interest to study such regularising effect of the noise in the discretisation of SDEs with irregular drift.
The most common approximation method is the Euler-Maruyama scheme
\begin{equ}\label{eq:euler}
dX_t^{n}=b(X_{k_{n}(t)}^{n})\,dt+\,dW_t,\quad X_0^{n}=x_0^n,
\end{equ}
where $k_n(t)=\lfloor nt\rfloor/n$.
When $b$ is Lipschitz continuous, then the analysis of the scheme is quite standard. 
Going beyond (locally) Lipschitz $b$, the first result goes back to \cite{GyK}, who established convergence in probability (without rate) of the Euler-Maruyama scheme.
In the recently revived interest in quantifying the rate of convergence, there has been essentially two classes of approaches:
\begin{enumerate}
\item[(A)] In \cite{MP, NT5, NT8, NT_MathComp, PT, Bao1, Bao2, Huang, Mik2}
mild assumptions on the modulus of continuity of $b$ are imposed.
This usually means $\alpha$-H\"older continuity (although \cite{Bao2} also discusses the Dini continuous case). 
While $\alpha>0$ is allowed to be arbitrarily low, the drawback is that the convergence rates obtained become increasingly worse as $\alpha\to 0$. In \cite{PT} for example, whose setting is similar to ours, the rate of $L_2$-convergence is $\alpha/2$, which becomes negligible for small $\alpha$.
\item[(B)] In \cite{Szo1,Szo2,Szo3,LS5,Y2, MY4}
$b$ is allowed to have discontinuities on a `small' (lower dimensional) set. In dimension $d=1$, for example, the set of discontinuities is a discrete set of points.
Outside of this exceptional set the usual Lipschitz condition in assumed. Under this condition, rate $1/2$ is achieved for Euler-type schemes, and in the scalar case recently the improved rate $3/4$ was shown \cite{MY4} for a modified Milstein-type scheme.
\end{enumerate}
Our contribution is twofold:
\begin{itemize}
\item Note that the results for (A) seem to suggest that the rate of convergence is arbitrarily small for $\alpha$-H\"older drift $b$, when $\alpha$ is small.
Here we show that this is \emph{not} the case: in Theorem \ref{thm:main-theorem} we show that (almost) $1/2$ rate of convergence holds for all Dini continuous coefficients.
\item In dimension $1$, we unify the type of irregularities considered (A) and (B) in the following sense: on one hand the irregularities do not have to be restricted on exceptional sets (as in (A)), on the other hand they are allowed to be discontinuous (as in (B)).
More precisely, in Theorem \ref{thm:main-theorem2}
we prove that for any $\eps>0$, the Euler-Maruyama scheme has rate $1/2-\eps$ in $L_2$ for all bounded and measurable $b$.
To the best of our knowledge this is the first result on the rate of convergence (and as far as $L_2$ is concerned, even merely on convergence) without posing any continuity assumption whatsoever on $b$.
\end{itemize}
We now proceed by stating our main results. Recall the definition of Dini continuity: fix a continuous increasing function $\vartheta : [0, 1] \to [0, \infty)$  such that 
\begin{equ}\label{eq:dini-cond}
\int_0^1 \frac{\vartheta(r)}{r} \, dr < \infty.
\end{equ}
We then denote by $\mathcal{D}$ the space of all continuous functions  $f:\bR^d\to \bR$, such that 
\begin{equs}
\| f\|_{\mathcal{D}}:= 
\sup_{x\in\R^d}|f(x)|+ \sup_{\substack{x,y\in\R^d\\|x-y| \leq 1\\ x\neq y}}  \frac{ |f(x)-f(y)|}{\vartheta(|x-y|)} < \infty.
\end{equs}
For vector-valued $f$, $\|f\|_\cD$ denotes $\max_i\|f^i\|_\cD$ (and similarly for other norms);
since the context will always make it clear whether we mean scalar- or vector-valued functions, this abuse of notation will not cause any confusion.
The main results then read as follows. 
\begin{theorem}                        \label{thm:main-theorem}
Assume $b\in\cD$.
Let $\eps \in (0,1) $ and suppose that for some constant $C$ one has $\E|x_0^n-x_0|^2 \leq C n^{-1+\eps}$ for all $n \in \bN$.
Then  for all $n \in \bN$ one has the bound
\begin{equs}\label{eq:main bound}
\sup_{t \in[0,T]}\E |X_t^n-X_t|^2  \leq N n^{-1+\eps}
\end{equs}
with some $N=N(C,\eps,\vartheta,T,d,\|b\|_{\mathcal{D}})$.
\end{theorem}
Let us mention that the generality of allowing different initial condition $x_0^n$ of the discretisation than that of \eqref{eq:main} is a convenient setup for the proof, see Section \ref{sec:proof}. In the next theorem we consider the scalar case where even the assumption of Dini continuity can be lifted.
\begin{theorem}                     \label{thm:main-theorem2}
Assume  $d=1$, and let $b:\R\to\R$ be a bounded measurable function. Let $\eps \in (0,1) $  and suppose that for some constant $C$ one has $\E|x_0^n-x_0|^2 \leq C n^{-1+\eps}$ for all $n \in \bN$.
Then  for all $n \in \bN$ one has the bound
\begin{equs}\label{eq:main bound2}
\sup_{t \in[0,T]}\E |X_t^n-X_t|^2  \leq N n^{-1+\eps}
\end{equs}
with some $N=N(C,\eps,T,\sup|b|)$.
\end{theorem}

\begin{remark}\label{rem:slow}
The choice of looking at $L_2$-convergence is also motivated by recent results \cite{HHJ, JMY, GJS, Y} that show that in $L_2$
very slow convergence rate(in particular, worse than any polynomial) of arbitrary approximation schemes may occur even in cases when the rate of convergence in probability is known to be $1/2$ \cite{Gy}.
Our results can also be seen as extending the class of equations where the slow convergence phenomena can \emph{not} happen.
\end{remark}

\begin{remark}
A few months after the current work was available on the ArXiv,
the article \cite{Neuenkirch} also appeared on the ArXiv.
In the one-dimensional setting the authors show an interesting complementary result to Theorem \ref{thm:main-theorem2}: if in addition to our assumption $b$ possesses regularity of order $\alpha \in (0,1)$ in a Sobolev-Slobodecki scale, then the rate of convergence also improves.
The technique in \cite{Neuenkirch} is essentially the same as the one of the present work in $d=1$: after a transformation by the scale function \eqref{eq:Zvonkin PDE} a key estimate of the type \eqref{eq:averaging 2} leads to the result; illustrating that, the method is reasonably robust.
 We would also like to thank the authors of \cite{Neuenkirch} for bringing to our knowledge that a time-homogeneous version of the estimate \eqref{eq:main-estimate} below has appeared in \cite{Altmeyer}.
\end{remark}

\section{Auxiliary Results}

\subsection{Quadrature bounds}
In this section we formulate some lemmata that will be used in the proofs of Theorems \ref{thm:main-theorem}-\ref{thm:main-theorem2}.
In the proofs the constants $N$ may change from line to line.
Our first lemma is similar in spirit to \cite[Proposition 2.1]{Davie},
which is the "basic estimate" in Davie's proof of the path-by-path uniqueness
of the solution of \eqref{eq:main}.
A time-homogeneous variant of the estimate below without stopping times and further random dependence also recently appeared in \cite{Altmeyer}, following a different proof.
 In the following, $\mathcal{Y}$ is a Polish space. 
 \begin{lemma}          \label{lem:main-lemma}
There exists a constant $N=N(T,d)$ such that for all stopping times $\tau \leq \tau' \leq T$, all $\mathcal{F}_{\tau}$-measurable $\mathcal{Y}$-valued random variables $Y$,  all 
bounded measurable functions $f:[0,T] \times \R^d\times \mathcal{Y} \to\R$, and all $n \in \mathbb{N}$,  one has the bound
\begin{equation}       \label{eq:main-estimate} 
\E \left|\int_{\tau}^{\tau'} \left( f(s,W_s, Y)- f(s, W_{k_n(s)}, Y ) \right) \, ds \right|^2 \leq N \big(\sup|f|^2\big) n^{-1} \log (n+1).
\end{equation}

\end{lemma}

\begin{proof}
Introduce the notation $W_{s,t}=W_t-W_s$.
We will first show that for $f: [0,T] \times \bR^d \to \bR$  and for $\alpha \in [0, T-2/n]$ we have the bound 
\begin{equation}       \label{eq:first-step. } 
2I_n:= \E \left|\int_{\alpha+1/n}^{T} \left( f(s,W_{\alpha,s})- f(s, W_{\alpha,k_n(s)}) \right) \, ds \right|^2 \leq N \big(\sup|f|^2\big) n^{-1} \log (n+1).
\end{equation}
We make two simplifying assumptions that will be removed later on: (i) $f(t,  \cdot)$ is compactly supported and twice continuously differentiable for all $t \in [0,T]$; (ii) $d=1$.
Notice that 
\begin{equs}              \label{eq:suf}
2I_n=\E \left|\int_{\alpha+1/n}^{T} \left( f(s,W_{s-\alpha})- f(s, W_{k_n(s)-\alpha}) \right) \, ds \right|^2,
\end{equs}
which gives,
\begin{equs}
I_n&=\E\int_{\alpha+1/n}^T \int_{\alpha+1/n}^t\left( f(s,W_{s-\alpha})- f(s, W_{k_n(s)-\alpha})\right)
\left( f(t, W_{t-\alpha})- f(t,W_{k_n(t)-\alpha})\right)\,ds\,dt
\\
&=\int_{\alpha+2/n}^{T}\int_{\alpha+1/n}^{k_n(t)-1/n}F_n(s,t) \,ds\,dt
\\
&\quad + \int_{\alpha+2/n}^{T}\int_{k_n(t)-1/n}^{t} F_n(s,t)\,ds\,dt
 + \int_{\alpha+1/n}^{\alpha+2/n}\int_{\alpha+1/n}^{t}F_n(s,t)\, ds \, dt 
\\  \label{eq:Ins}  
&=: I_n^1+I_n^2+I_n^3,
\end{equs}
where we have used the notation 
\begin{equs}
F_n(s,t):= \E \big[ \left( f(s,W_{s-\alpha})- f(s, W_{k_n(s)-\alpha})\right)
\left( f(t, W_{t-\alpha})- f(t,W_{k_n(t)-\alpha})\right)\big] .
\end{equs}
We deal first with $I_n^1$. We will show that for $(s,t)$ as in the integrand of $I_n^1$ we have the bound   
\begin{equs}             \label{eq:F1-F2}
 | F_n(s, t)| \leq  N n^{-1}(\sup|f|^2) \big( \big((k_n(s)-\alpha)(k_n(t)-s)\big)^{-1/2}+ (k_n(t)-s)^{-1}\big).
\end{equs}
By  the fundamental theorem of calculus
\begin{equs}
 F_n(s, t)=&  \E \int_0^1 \int_0^1 (W_{s-\alpha}-W_{k_n(s)-\alpha})f'(s, W_{s-\alpha}- \theta(W_{s-\alpha} -W_{k_n(s)-\alpha}) ) 
\\
& \qquad \times  (W_{t-\alpha}-W_{k_n(t)-\alpha})f'(t, W_{t-\alpha}- \lambda (W_{t-\alpha}-W_{k_n(t)-\alpha}) ) \,  ds\, dt\,  d\theta\, d \lambda
\end{equs}
Notice that for $s< k_n(t)$, we have 
$k_n(s)\leq s<k_n(t)\leq t$.
We can express the integrand in terms of random variables whose joint density is relatively simple:
Setting
\begin{equ}
Y_1:=W_{k_n(s)-\alpha},\ Y_2:=W_{s-\alpha}-W_{k_n(s)-\alpha},\ Y_3:=W_{k_n(t)-\alpha}-W_{s-\alpha},\ Y_4:=W_{t-\alpha}-W_{k_n(t)-\alpha}
\end{equ}
and introducing the notations, for $x=(x_1,x_2,x_3,x_4) \in \bR^4$,
\begin{equs}
\bar{f}(s,x)=f(s,(1-\theta)x_2+x_1), \qquad \tilde{f}(t,x)=f(t,(1- \lambda)x_4+x_3+x_2+x_1),
\end{equs}
we can write 
\begin{equs}
f' (s,W_{s-\alpha}- \theta(W_{s-\alpha}-W_{k_n(s)-\alpha}) )&= \bar f_{x_1}(s,Y),\\
f^\prime (t,W_{t-\alpha}- \lambda(W_{t-\alpha}-W_{k_n(t)-\alpha}) )&= \tilde f_{x_3}(s,Y).
\end{equs}
Moreover $Y_i=Y_i(s,t)$ are independent, Gaussian, with mean $0$ and variance $\sigma_i^2=\sigma_i^2(s,t)$, where
\begin{equs}
\sigma^2_1(s,t)= k_n(s)-\alpha, \ \sigma^2_2(s,t)= s- k_n(s), \ \sigma^2_3(s,t)= k_n(t)-s, \ \sigma^2_4(s,t)=t- k_n(t).
\end{equs}
Their joint density is therefore given by
\begin{equs}
\rho(s,t,x)= \exp\left(-\sum_{i=1}^4 \frac{|x_i|^2}{2\sigma^2_i(s,t)}\right) \frac{1}{(2\pi)^{2} \prod_{i=1}^4 \sigma_i(s,t) },
\end{equs}
for $x=(x_1,x_2,x_3,x_4) \in \bR^4$.
Hence we can write
\begin{equs}
 F_n(s, t)= \int_0^1\int_0^1\int_{\bR^4}  x_4 x_2 \bar{f}_{x_1}(s,x)\tilde{f}_{x_3}(t,x) \rho(s,t,x)\, dx   d\theta d \lambda,
\end{equs}
where 
\begin{equs}
\bar{f}(s,x)=f(s,(1-\theta)x_2+x_1), \qquad \tilde{f}(t,x)=f(t,(1- \lambda)x_4+x_3+x_2+x_1).
\end{equs}
After integration by parts with respect to $x_1$ and $x_3$ we get
\begin{equs}
 F_n(s, t)=& \int_0^1 \int_0^1\int_{\bR^4}  x_4 x_2 \bar{f}(s,x)\tilde{f}(t,x) \rho_{x_1x_3}(s,t,x)\, dx  d\theta d \lambda
\\
+ & \int_0^1 \int_0^1\int_{\bR^4}  x_4 x_2 \bar{f}(s,x)\tilde{f}_{x_1}(t,x) \rho_{x_3}(s,t,x)\, dx  d\theta d \lambda
\\
=&: F^1_n(s,t)+F_n^2(s,t).
\end{equs}
By explicitly differentiating the Gaussian density $\rho$, we get 
\begin{equs}
| F^1_n(s,t)| & = \big| \int_0^1 \int_0^1   \int_{\bR^4}  x_4 x_2 \bar{f}(s,x)\tilde{f}(t,x)  \rho(s,t,x) x_1 x_3 
 \sigma_1^{-2}(s,t)\sigma_3^{-2}(s,t)\, dx  d\theta d \lambda\big|
\\
& \leq (\sup|f|^2) \left(  \prod_{i-1}^4\E |Y_i(s,t)|\right)  \sigma_1^{-2}(s,t)\sigma_3^{-2}(s,t)
\end{equs}
Therefore, since $\E|Y_i(s,t)|\leq \sigma_i(s,t)$ and $\sigma_2(s,t),\sigma_4(s,t)\leq n^{-1/2}$, we can write
\begin{equs}   \label{eq:F1} 
| F^1_n(s,t)| \leq N n^{-1}K^2 \big((k_n(s)-\alpha)(k_n(t)-s)\big)^{-1/2}.
\end{equs}
For $F^2_n(s, t)$, notice that $\tilde f_{x_1}=\tilde f_{x_3}$, so we can integrate by parts with respect to $x_3$ again to get
\begin{equs}
|F^2_n(s, t)|=& \big|\int_0^1 \int_0^1 \int_{\bR^4}  x_4 x_2 \bar{f}(s,x)\tilde{f}(t,x) \rho_{x_3x_3}(s,t,x)\, dx  d\theta d \lambda\big|
\\
\leq &\big|\int_0^1 \int_0^1 \int_{\bR^4}  x_4 x_2 \bar{f}(s,x)\tilde{f}(t,x) |x_3|^2 \sigma^{-4}_3(s,t)\rho(s,t,x)\, dx  d\theta d \lambda\big|
\\
& +\big| \int_0^1 \int_0^1\int_{\bR^4}  x_4 x_2 \bar{f}(s,x)\tilde{f}(t,x) \sigma^{-2}_3(s,t)\rho(s,t,x)\, dx  d\theta d \lambda\big|
\\
 \leq  & N n^{-1}(\sup|f|^2) (k_n(t)-s)^{-1},
\end{equs}
which combined with \eqref{eq:F1} shows \eqref{eq:F1-F2}.
Consequently, we have
\begin{equs}
I_n^1\leq  N n^{-1}K^2 \int_{\alpha+2/n}^{T}\int_{\alpha+1/n}^{k_n(t)-1/n}
  \big( \big((k_n(s)-\alpha)(k_n(t)-s)\big)^{-1/2}+ (k_n(t)-s)^{-1}\big) \,ds\,dt.
\end{equs}
Note that one has $k_n(s)-\alpha \geq s-\alpha -1/n$ and that by a change of variables one sees
\begin{equ}
\int_{
\alpha+1/n}^{k_n(t)} \big((s-\alpha -n^{-1})(k_n(t)-s)\big)^{-1/2}\,ds=\int_0^1\big(s(1-s)\big)^{-1/2}\,ds=\pi.
\end{equ}
Next, we see that
\begin{equs}
  \int_{\alpha+2/n}^T \int_{\alpha +1/n}^{k_n(t)-1/n} (k_n(t)-s)^{-1} ds dt
& =   \int_{\alpha+2/n}^T \log (k_n(t)-\alpha -n^{-1})- \log (n^{-1}) \, dt 
 \\
&  \leq N  \log (n+1).
\end{equs}
Therefore, we get 
\begin{equs}         \label{eq:I1}
|I^{1}_n| \leq N  K^2 n^{-1}\log (n+1).
\end{equs}
Both $I^2_n$ and $I^3_n$ are integrals over domains whose size is bounded by $2Tn^{-1}$ with an integrand that is bounded by $\sup|f|^2$. 
Consequently, 
\begin{equs}  \label{eq:I2}
|I^2_n|+|I^3_n| \leq N  n^{-1}(\sup|f|^2).
\end{equs}
Combining  \eqref{eq:Ins} with \eqref{eq:I1}-\eqref{eq:I2} we obtain \eqref{eq:first-step. }  under the additional assumptions (i)-(ii). We proceed by removing first the additional assumption (ii).  Suppose now that we have \eqref{eq:first-step. }  for $d=k$ and let   $W=(W^1, . . ., W^{k+1})$ be a $k+1$-dimensional Wiener process.
Setting $\tilde{W}=(W^2,..., W^{k+1})$,
we have 
\begin{equs}       
  \E  \big| &|\int_{\alpha+{1/n}}^T  \big( f(s, W_{\alpha,s})- f(s, W_{\alpha,k_n(s)}) \big) \, ds \big|^2
\\
& =  \E   \big|\int_{\alpha+{1/n}}^T \big( f(s,  W^1_{\alpha,s}, \tilde{W}_{\alpha,s})- f(s,  W^1_{\alpha,k_n(s)}, \tilde{W}_{\alpha,k_n(s)}) \big) \, ds \big|^2
\\
& \leq   \E  \Big(  \E \Big[ \big|\int_{\alpha+{1/n}}^T \big( f(s, W^1_{\alpha,s} , \tilde{W}_{\alpha,s} )
 - f(s,  W^1_{\alpha,s} , \tilde{W}_{\alpha,k_n(s)} ) \big) \, ds \big|^2 \Big| W^1 \Big]\Big)
\\
&\quad 
+ \E  \Big(  \E \Big[ \big|\int_{\alpha+{1/n}}^T \big( f(s, W^1_{\alpha,s} , \tilde{W}_{\alpha,k_n(s)} )
 - f(s,  W^1_{\alpha,k_n(s)} , \tilde{W}_{\alpha,k_n(s)} ) \big) \, ds \big|^2 \Big| \tilde W \Big]\Big)
\\
& \leq N (\sup|f|^2) n^{-1} \log (n+1),
\end{equs}
where we have used the fact that $W^1, \tilde{W}$ are independent and the induction hypothesis.

Removing the additional regularity assumption (i) on $f$ follows from a standard approximation argument. For a bounded measurable $f$ take a sequence $(f_m)_{m\in\N}$ satisfying (i), such that for each $t \in [0,T]$,  $f_m(t,  \cdot)\to f(t,  \cdot)$ as $m\to \infty$ almost everywhere on $\bR^d$, and $\sup|f_m|\leq\sup|f|$ (one can construct such $f_m$ by, for example, mollification).
Since for positive times the law of the Brownian motion is absolutely continuous with respect to the Lebesgue measure, we have that for all $s\in [\alpha+2/n, T] $, $f_m(s,W_s-W_\alpha)\to f(s,W_s-W_\alpha)$  and $f_m(s,W_{\kappa_n(s)}-W_\alpha)\to f(s,W_{\kappa_n(s)}-W_\alpha)$ almost surely. Let $\alpha \in [0,T-2/n]$.  By the first part of the proof we have the inequality
\begin{equs}
\E \left|\int_{\alpha + 2/n}^T \left( f_m(s,W_s-W_\alpha)- f_m(s,W_{k_n(s)}-W_\alpha) \right) \, ds \right|^2 &\leq N \sup|f_m|^2 n^{-1} \log (n+1)
\\&\leq N \sup|f|^2 n^{-1} \log (n+1),
\end{equs}
 and  by Fatou's lemma we get 
\begin{equs}
\E \left|\int_{\alpha+2/n}^T \left( f(s,W_s-W_\alpha)- f(s,W_{k_n(s)}-W_\alpha) \right) \, ds \right|^2 &\leq N (\sup|f|^2) n^{-1} \log (n+1).
\end{equs}
It then immediately follows that for any $\alpha\in[0,T]$ we have the bound
\begin{equs}        
\E \left|\int_{\alpha}^T \left( f(s,W_s-W_\alpha)- f(s,W_{k_n(s)}-W_\alpha) \right) \, ds \right|^2 &\leq N (\sup|f|^2) n^{-1} \log (n+1).
 \label{eq:estimate-alpha}
\end{equs}
Now let us consider a simple stopping time $\tau \leq T$, that is, one that takes only finitely many values $\{ \alpha_1, ..., \alpha_m\}$. Take an $\mathcal{F}_\tau$-measurable random variable $Y$, and a measurable bounded function $f : [0,T] \times \bR^d \times \mathcal{Y} \to \bR$. 
We have 
\begin{equs}
 \E \Big|\int_{\tau }^T &\big( f(s,W_s, Y)- f(s,W_{k_n(s)}, Y) \big) \, ds\Big|^2 
\\
= &  \sum_{i=1}^m \E\Big( \one_{\tau=\alpha_ i } \E \Big( \Big|\int_{\alpha_i }^T \big( f(s,W_s, Y)- f(s,W_{k_n(s)}, Y) \big) \, ds\Big|^2 \Big| \mathcal{F}_{\alpha_i} \Big).
\end{equs}
Since 
\begin{equs}
 \one_{\tau=\alpha_ i }\E \Big( \Big|\int_{\alpha_i }^T \big( f(s,W_s, Y)- f(s,W_{k_n(s)}, Y) \big) \, ds\Big|^2 \Big| \mathcal{F}_{\alpha_i} \Big)= \one_{\tau=\alpha_ i }g(W_{\alpha_i}, Y), 
\end{equs}
with 
\begin{equs}
g(x, y)= \E \Big|\int_{\alpha}^T \left( f(s,W_s-W_\alpha+x, y )- f(s,W_{k_n(s)}-W_\alpha+x, y) \right) \, ds \Big|^2, 
\end{equs}
it follows from \eqref{eq:estimate-alpha} that 
\begin{equs}\label{eq:almost done}
\E \Big|\int_{\tau }^T \big( f(s,W_s, Y)- f(s,W_{k_n(s)}, Y) \big) \, ds\Big|^2 \leq N \sup|f|^2 n^{-1} \log (n+1). 
\end{equs}
If $\tau \leq T$ is an arbitrary stopping time, one can find  
a sequence of simple stopping times $\tau_m$ converging to $\tau$ from above which by Fatou's lemma implies that (notice that if $Y$ is $\mathcal{F}_\tau$-measurable then it is $\mathcal{F}_{\tau_m}$-measurable as well) \eqref{eq:almost done} holds for arbitrary stopping times $\tau \leq T$. 
Finally, \eqref{eq:main-estimate} follows from \eqref{eq:almost done} since $\int_{\tau}^{\tau'}=\int_{\tau}^T-\int_{\tau'}^T$.
\end{proof}

\begin{lemma}                \label{lem:main-lemma-2}
Let $b:\R^d\to\R^d$ and $f:[0,T]\times\R^d\times\mathcal{Y}\to\R$ be bounded measurable functions and define $X^n$ as in \eqref{eq:euler}.
Let $\tau \leq \tau' \leq T$ be stopping times and $Y$ be a $\mathcal{F}_{\tau}$-measurable $\mathcal{Y}$-valued random variable.
Then for all $\eps\in(0,1)$ one has the bound
\begin{equs}\label{eq:averaging 2}
\E \Big|\int_{\tau}^{\tau'} \big( f(s,X^n_s,Y)- f(s, X^n_{k_n(s)},Y) \big) \, ds \Big|^2 \leq N  n^{-1+\eps},
\end{equs}
for all $n \in \bN$, $r \in [0,T]$, where
$N=N(\eps,T,d,\sup|b|,\sup|f|)$.
\end{lemma}
\begin{proof}
For an $\R^d$-valued stochastic process $Z$ let us denote
\begin{equ}
h(Z)=\Big|\int_{\tau}^{\tau'} \big( f(s,Z_s,Y)- f(s, Z_{k_n(s)},Y) \big) \, ds \Big|.
\end{equ}
In other words, we aim to bound $\E h(X^n)^2$. Notice also that \eqref{eq:main-estimate} provides a bound for $\E h(W+x)^2$ whenever $W$ is an $(\mathcal{F}_t)_{t \in [0,T]}$-Wiener process and $x\in\R^d$, as well as the fact that one has a trivial bound $h\leq N$.
Let us set 
\begin{equs}
\rho_n = \exp\Big(-\int_0^T \sum_{i=1}^nb^i(X^n_{k_n(s)}) \, dW^i_s - \frac{1}{2}\int_0^T \sum_{i=1}^d |b^i(X^n_{k_n(s)})|^2 \, ds  \Big). 
\end{equs}
By Girsanov's theorem, under the measure $d \bP^n = \rho_n d \bP$, $X^n-x$ is 
an  $(\mathcal{F}_t)_{t \in [0,T]}$-Wiener process, therefore by the above properties and H\"older's inequality we have
\begin{equs}
\E h(X^n)^2	& \leq N \E\big(h(X^n)^{2-\eps}\rho_n^{(2-\eps)/2} \rho_n^{(\eps- 2)/2}\big)
\\
&\leq N[\E\big(h(X^n)^2\rho_n\big)]^{(2-\eps)/2}\,[\E\rho_n^{(\eps- 2)/\eps}]^{\eps/2}
\\
&\leq N [ n^{-1} \text{log} \  (n+1) ]^{(2-\eps)/2}\,[\E\rho_n^{(\eps- 2)/\eps}]^{\eps/2}.
\end{equs}
Notice that $ [ n^{-1} \text{log} \  (n+1) ]^{(2-\eps)/2} \leq  N n^{-1+\eps}$,
and we are done since
\begin{equs}
\sup_n \E  \ \rho_n^{(\eps- 2)/\eps}  \leq N \exp\Big( \tfrac{(\eps-2)^2}{\eps^2}T \sup|b|\Big)\leq N.
\end{equs}
\end{proof}

\subsection{Regularity of the Kolmogorov equation}
The final ingredient concerns the regularity of the associated Kolmogorov equation.
While the result is possibly known, we did not find a reference in this form, so we provide a short proof.
We note that \eqref{eq:PDE1} is the only instance where the Dini continuity assumption is used.
We denote by $C^{1,2}(Q_T)$ the space of all bounded continuous functions $f$  on $Q_T=(0,T) \times \bR^d$ such that the derivatives $\D_tf$, $\nabla f$,$\nabla^2f$ exist, are continuous and bounded, and we use the norm
\begin{equs}
\|f\|_{C^{1,2}(Q_T)}:= \|f\|_{L_\infty(Q_T)}+\|\nabla f\|_{L_\infty(Q_T)}+\|\nabla^2f\|_{L_\infty(Q_T)}+\|\D_t f\|_{L_\infty(Q_T)}.
\end{equs}

\begin{lemma} \label{lem:PDE-estimates1}
Let $f,g\in \mathcal{D}$ and suppose that $\|f\|_{\mathcal{D}} \leq K$ for some $K\geq 0$.  There exists $T_0>0 $  depending only on $K$ and $d$, such that  there exists a unique bounded classical solution $u$ of
\begin{equs}
\begin{aligned}                     \label{eq:main-PDE}
\D_tu& = \tfrac{1}{2}\Delta u+f\cdot\nabla u +g, \qquad u(0,\cdot)=0.
\end{aligned}
\end{equs}
on $Q_{T_0}$. Moreover, for all $T\in(0,T_0]$, the solution of \eqref{eq:main-PDE} on $[0,T]$ satisfies the bounds
\begin{equs}                   
\| u\|_{C^{1,2}(Q_{T})} &\leq N \|g \|_{\mathcal{D}},\label{eq:PDE1}
\\            
\|\nabla u^i \|_{L_\infty(Q_T)} &\leq \sqrt{T} N \| g\|_{L_\infty(\R^d)},\label{eq:PDE2}
\end{equs}
where $N=N(d,\vartheta,K)$. 
\end{lemma}

To prove Lemma \ref{lem:PDE-estimates1}, since replacing $\vartheta(r)$ by $\vartheta(r)\vee \sqrt{r}$ can only decrease the $\|\cdot\|_\cD$ norm of any function,
without loss of generality we assume that $\vartheta(r)\geq \sqrt{r}$ for all $r\in[0,1]$. We denote by $\cD_T$ the space of all continuous functions on $[0,T]\times \bR^d$ such that 
$$
\| f\|_{\cD_T} := \sup_{t \in [0,T]} \| f(t)\|_{\cD} < \infty.
$$
First consider the simpler equation 
\begin{equation}                      \label{eq:no-grad}
\D_tu=\tfrac{1}{2}\Delta u+f ,  \qquad  
u(0, \cdot)=0.            
\end{equation}
We will use the following well-known properties of \eqref{eq:no-grad}.
While these are well-known in the PDE folklore (for an elliptic counterpart, see e.g. \cite{Gilbarg-Trudinger}), we did not find a reference for this exact form, so we sketch the proof in the appendix for the interested reader.
\begin{lemma}\label{lem:PDE easy}
Let $T \in (0,1]$ and $f \in \mathcal{D}_T$. Let $u$ be a bounded distributional solution of \eqref{eq:no-grad}. Then $u$ is a classical solution of \eqref{eq:no-grad} which moreover satisfies for $\alpha\in(0,1)$
\begin{equs}                       \label{eq:dini-no-grad}
\| u\|_{C^{1,2}(Q_T)} \leq & N_0 \| f\|_{\mathcal{D}_T},
\\                        \label{eq:est-u}
\| u\|_{L_\infty(Q_T)} \leq & N_1 T \| f\|_{L_\infty(Q_T)},
\\                       \label{eq:est-grad-Holder}
\sup_{t \in [0,T]}\|\nabla u(t)\|_{C^\alpha(\bR^d)}  \leq & N_2 T^{(1-\alpha)/2} \| f\|_{L_\infty(Q_T)}
\end{equs}
where $N_0=N_0(d,  \vartheta)$, $N_1=N_1(d)$, and $N_2=N_2(d)$.
\end{lemma}

\begin{proof}[Proof of Lemma \ref{lem:PDE-estimates1}]
Uniqueness easily follows, see e.g. \cite{Krylov}.
As for existence,
let $g_n,   f_n $   be bounded functions with bounded derivatives of any order such that 
\begin{equs}
\|g_n-g\|_{L_\infty(\bR^d)}+\|f_n-f\|_{L_\infty(\bR^d)} \to 0, \qquad  \text{as} \ n \to \infty,
\end{equs}
and 
$$
\|g_n\|_{L_\infty(\bR^d)}\leq \|g\|_{L_\infty(\bR^d)}, \qquad \|f_n\|_{L_\infty(\bR^d)} \leq \|f\|_{L_\infty(\bR^d)}.
$$
The problem 
\begin{equs}
\D_tu_n = \tfrac{1}{2}\Delta u_n+f_n\cdot \nabla u_n +g_n,  \qquad  
u(0, \cdot)=0         
\end{equs}
has a classical solution $u_n \in  C^{1,2}(Q_{T_0})$ for any $T_0>0$ such that $u_n,\nabla u_n\in C(\overline{Q_{T_0}})$ (see e.g. \cite{Krylov}).
By \eqref{eq:est-grad-Holder} we have
\begin{equs}
\sup_{t \in [0,T_0]} \|\nabla u_n(t)\|_{C^{1/2}(\bR^d)} \leq N_2(d) {T_0}^{1/4} \left( \| f\|_{L_\infty(\bR^d)} \| \nabla u_n\|_{L_\infty(Q_T)}+\|g\|_{L_\infty(\bR^d)}\right). 
\end{equs} 
Hence, for  sufficiently small ${T_0}>0$ depending only on $d$ and $K$, we have
 for all $n \in \mathbb{N}$
\begin{equs}       \label{eq:bound-gradients}
\sup_{t \in [0,T_0]} \|\nabla u_n(t)\|_{C^{1/2}(\bR^d)} \leq \|g\|_{L_\infty(\bR^d)}. 
\end{equs} 
Similarly, by \eqref{eq:est-u}-\eqref{eq:est-grad-Holder}, for ${T_0}$ sufficiently small,  we see that 
\begin{equs}
\|u_n- & u_m\|_{L_\infty(Q_{T_0})}+ \sup_{t \in [0,{T_0}]}\| \nabla u_n(t)-\nabla u_m(t)\|_{C^{1/2}(\bR^d)} 
\\
&\leq \|g_n-g_m\|_{L_\infty(\bR^d)}+\|f_n-f_m\|_{L_\infty(\bR^d)} \|\nabla u_m\|_{L_\infty(Q_{T_0})}
\\
&\leq  \|g_n-g_m\|_{L_\infty(\bR^d)}+\|f_n-f_m\|_{L_\infty(\bR^d)} \|g\|_{L_\infty(\bR^d)},
\end{equs}
where for the last inequality we used \eqref{eq:bound-gradients}. Consequently, $u_n$ converges to a limit $u \in C(\overline{Q_{T_0}})$ with $\nabla u \in C(\overline{Q_{T_0}})$,  
which is a distributional solution \eqref{eq:main-PDE}. Moreover, it satisfies
\begin{equs}       \label{eq:est-grad-u-holder}
\sup_{t \in [0,T_0]}\| \nabla u(t) \|_{C^{1/2}(\bR^d)} \leq  \|g\|_{L_\infty(\bR^d)}.
\end{equs}
Consequently, $f \nabla u \in \mathcal{D}_{T_0}$, which implies by Lemma \ref{lem:PDE-estimates1} that $u$ is a classical solution of \eqref{eq:main-PDE} and satisfies
\begin{equs}
\| u\|_{C^{1,2}(Q_{T_0})} \leq & N \left( \| f \nabla u\|_{\mathcal{D}_{T_0}} +\|g\|_{\mathcal{D}} \right)
\\
\leq & N \left( \| f \|_{\mathcal{D}}\|\nabla u\|_{\mathcal{D}_{T_0}} +\|g\|_{\mathcal{D}} \right)
\\
\leq & N\left( \| f \|_{\mathcal{D}}\sup_{t \in [0,{T_0}]}\|\nabla u\|_{C^{1/2}(\bR^d)} +\|g\|_{\mathcal{D}} \right),
\end{equs}
where $N= N(\vartheta, d)$ and we have used $\vartheta(r)\geq \sqrt{r}$ for the last inequality. By \eqref{eq:est-grad-u-holder} and \eqref{eq:dini-no-grad} we obtain 
\begin{equs}
\| u\|_{C^{1,2}(Q_{T_0})} \leq  N \|g\|_{\mathcal{D}},
\end{equs}
with $N=N(\vartheta, d, K )$, so the estimates \eqref{eq:PDE1}-\eqref{eq:PDE2} hold for $T=T_0$.

The bound \eqref{eq:PDE1} trivially extends to $T<T_0$.
As for \eqref{eq:PDE2},
with $\lambda = \sqrt{T}/ \sqrt{T_0}\leq 1$ let us set $u_\lambda (t,x):=u(\lambda^2 t, \lambda x)$. Clearly $u_\lambda$ satisfies on $Q_{T_0}$
\begin{equs}
\D_tu_\lambda& = \tfrac{1}{2}\Delta u_\lambda+\lambda f_\lambda\cdot\nabla u_\lambda + \lambda^2 g_\lambda, \qquad u_\lambda(0,\cdot)=0,
\end{equs}
with $f_\lambda(x)=f(\lambda x)$ and $g_\lambda(x)=g(\lambda x)$. Since 
\begin{equ}
\| \lambda f_\lambda \|_{\mathcal{D}} \leq \| f_\lambda \|_{\mathcal{D}} \leq \|f\|_{\mathcal{D}} = K,
\end{equ}
we get by applying \eqref{eq:PDE2} with $T=T_0$
\begin{equs}
\lambda\|\nabla u\|_{L_\infty(Q_T)}=\| \nabla u_\lambda \|_{L_\infty(Q_{T_0})} \leq N \lambda^2 \|g_\lambda \|_{L_\infty(\R^d)}=N\lambda^2\|g\|_{L_\infty(\R^d)},
\end{equs}
thus obtaining  \eqref{eq:PDE2} for arbitrary $T\in(0,T_0]$.

\end{proof}

\section{Proofs of the main results}\label{sec:proof}

\begin{proof}[Proof of Theorem \ref{thm:main-theorem}]                  
First we prove for sufficiently small (but $n$-independent) time horizon $T$. Let $T_0=T_0(\|b\|_\cD,d)$ be given by Lemma \ref{lem:PDE-estimates1}.
We invoke the idea of \cite{FGP} (see also \cite{PT}): for any $T \in (0,T_0]$ and for each $i =1,...d$,  by Lemma \ref{lem:PDE-estimates1} and a simple time reversal, there exists a classical solution $u^i\in C^{1,2}(Q_T) \cap C ( \overline{Q_T})$   to
\begin{equs}
\D_t u^i +\tfrac{1}{2}\Delta u^i +b\cdot \nabla u^i= -b^i, \qquad  u(T,\cdot)=0,
\end{equs}
which satisfies the bounds
\begin{equ}\label{eq:regularity u}
\|u^i\|_{C^{1,2}}\leq N'\|b\|_\cD
\end{equ}
\begin{equs}                   \label{eq;grad-u}               
\|\nabla u^i \|_{L_\infty(Q_T)} \leq \sqrt{T} N' \| b\|_{L_\infty(\R^d)},
\end{equs}
where $N'= N(d, \vartheta, \|b\|_\cD)$.  We emphasise that $N'$ is independent of $T$, which is a convention that we keep for the rest of the proof.
Denote the $j$-th partial derivative of $u$ by $u_{x^j}$. Applying It\^o's formula for
$u^i(t,X_t)$ we have 
\begin{equs}                \label{eq:Ito-Tanaka1}
\int_0^t b^i(X_s) \, ds = u^i(0, x_0)-u^i(t, X_t)+ \sum_{j=1}^d\int_0^t  u^i_{x_j} (s, X_s) \, dW^j_s.
\end{equs} 
Similarly, we have 
\begin{equs}   \label{eq:Ito-Tanaka2}
\int_0^t b^i(X^n_s) \, ds =& u^i(0, x_0^n)-u^i(t, X^n_t)+ \sum_{j=1}^d\int_0^t  u^i_{x_j} (s, X^n_s) \, dW^j_s
\\
&+ \sum_{j=1}^d\int_0^t u^i_{x_j}(s, X^n_s)\big( b^j(X^n_{k_n(s)}) -  b^j(X^n_s) \big) \, ds.
\end{equs} 
Moreover
\begin{equs}
\E |X_t-X^n_t|^2  
 &\leq  9 \E|x_0-x_0^n|^2+ 9\E\left|\int_0^t b(X_s) \, ds - \int_0^t b(X^n_s) \, ds \right|^2 
\\
&\qquad +9\E  \left|\int_0^t b(X^n_s) \, ds - \int_0^t b(X^n_{k_n(s)}) \, ds \right|^2 
\\
&= :I^1_n+I_n^2+I_n^3.
\end{equs}
For $I_n^2$ we have by \eqref{eq:Ito-Tanaka1}-\eqref{eq:Ito-Tanaka2}
\begin{equs}        
I_n^2 &\leq  N' \sum_{i=1}^d  \E|u^i(0, x_0)-u^i(0, x_0^n)|^2+N' \sum_{i=1}^d  \E|u^i(t, X_t)-u^i(t, X^n_t)|^2
 \\
&\quad+ \sum_{i,j=1}^d N' \E\left| \int_0^t  u^i_{x_j} (s, X_s)-u^i_{x_j} (s, X^n_s)   \, dW^j_s \right|^2 
\\
&\quad +N' \sum_{i,j=1}^d \E  \left| \int_0^t u^i_{x_j}(s, X^n_s) b^j(X^n_s) - u^i_{x_j}(s, X^n_{k_n(s)}) b^j(X^n_{k_n(s)}) ) \, ds \right|^2 
\\
&\quad+ N' \sum_{i,j=1}^d  \E \int_0^t |b^j(X^n_{k_n(s)}) )| ^2 |u^i_{x_j}(s, X^n_s) - u^i_{x_j}(s, X^n_{k_n(s)})|^2 \, ds
\\
&=:\sum_{i=1}^5I_n^{2i}.\label{eq:estimateC2}
\end{equs}
One then has the following estimates: by \eqref{eq:regularity u} and the assumption on $x_0$, $x_0^n$, we have
\begin{equ}
I_n^1+I_n^{21}\leq Nn^{-1+\eps};
\end{equ}
by \eqref{eq;grad-u} we have
\begin{equs}
I_n^{22}\leq N'\sqrt{T}\E |X_t-X^n_t|^2;
\end{equs}
by It\^o's isometry and \eqref{eq:regularity u} we have
\begin{equ}
I_n^{23}=N'\sum_{i,j=1}^d\int_0^t|u^i_{x_j} (s, X_s)-u^i_{x_j} (s, X^n_s) |^2  \, ds\leq N\int_0^t \E |X_s-X^n_s|^2 \, ds;
\end{equ}
by the boundedness of $b$ and by \eqref{eq:regularity u} we have
\begin{equs}
I_n^{25} \leq  & N  \int_0^t \E |X^n_s-X^n_{k_n(s)} |^2 \, ds \leq N n^{-1}.  
 \end{equs}
The terms $I_n^{24}$ and $I_n^3$ both can be estimated using Lemma \ref{lem:main-lemma-2}, with $f= u^i_{x_j}b^j$ and $f=b$, respectively, to get
\begin{equ}
I_n^{24}+I_n^3\leq N n^{-1+\eps}.
\end{equ}
Combining the above bounds we get
\begin{equs}
\E |X_t-X^n_t|^2  
\leq  N n^{-1+\eps} + N'\sqrt{T}\E |X_t-X^n_t|^2+  N\int_0^t \E |X_s-X^n_s|^2 \, ds.
\end{equs}
For sufficiently small $T$ (say $\sqrt{T}< 1/(2N')$) the second term on the right-hand side can be omitted at the price of a constant factor. The conclusion then follows from Gronwall's lemma applied to the function $t\mapsto\E|X_t-X^n_t|^2$.

We now remove the smallness assumption on $T$. By the above, for some $\delta=\delta(C,\eps,\vartheta,d,\|b\|_{\cD})>0$, the statement is proved for $T'$ in place of $T$, for arbitrary $T'\in [0,\delta]$.
Now fix $n$ and define $k$ by $k/n=\kappa_n(\delta)$. Denote $m=\lceil\tfrac{T}{\kappa_n(\delta)}\rceil$ and note that $m\leq \tfrac{2T}{\delta}\leq N$.
Using the above proven bound, we have
\begin{equ}\label{eq:rev0}
\sup_{t\in[0,k/n]} \E |X^n_t-X_t|^2\leq Nn^{-1+\eps}.
\end{equ}
On the time interval $[k/n,2k/n]$, $X_{t+k/n}$ is the solution of
\begin{equ}
dY_t=b(Y_t)\,dt+\,dW_{t+k/n},\quad Y_{0}=X_{k/n},
\end{equ}
while $X^n_{t+k/n}$ is the solution of 
\begin{equ}
dY^n_t=b(Y^n_{\kappa_n(t)})\,dt+\,dW_{t+k/n},\quad Y^n_0=X^n_{k/n}.
\end{equ}
Moreover, the initial conditions $X_{k/n}, X^n_{k/n}$ satisfy the condition of Theorem \ref{thm:main-theorem}, by \eqref{eq:rev0}. Since the theorem is now proven to hold for small times, we can conclude
\begin{equ}
\sup_{t\in[k/n,2k/n]} \E |X^n_t-X_t|^2=\sup_{t\in[0,k/n]} \E |Y^n_t-Y_t|^2\leq Nn^{-1+\eps},
\end{equ}
and combining with \eqref{eq:rev0},
\begin{equ}\label{eq:rev1}
\sup_{t\in[0,2k/n]} \E |X^n_t-X_t|^2\leq Nn^{-1+\eps}.
\end{equ}
Iterating this procedure $m$ times we reach the full time horizon $[0,T]$.
\end{proof}

\begin{proof}[Proof of Theorem \ref{thm:main-theorem2}] 
We again prove the statement for $T\in(0,T_0]$ for some sufficiently small $T_0>0$ (to be determined later), the general case can then be deduced exactly as above. Correspondingly, from here on $N=N(C,\eps,\sup|b|)$.

For $ z \in \mathbb{\bR}$ let us define
\begin{equs}
\phi_z(x):= \int_0^x  \exp \left( - 2 \int_z^{ r } I_{ |z-s| \leq  2}b(s) \, ds \right) \, dr, \qquad \psi_z(x):= \phi^{-1}_z(x).
\end{equs}
Notice that $ \phi_z \in W^{2, \infty}(\bR)$ and for almost everywhere on $  [z-2, z+2]$ we have
\begin{equ}\label{eq:Zvonkin PDE}
\tfrac{1}{2}\phi_z''+b\phi_z'=0.
\end{equ}
Moreover, it  is straightforward to check that
\begin{equs}                   \label{eq:all-Lip}
\sup_{z,x} \left( |\phi_z'(x)| + |\phi_z''(x)|+|\psi_z'(x)|+|(\phi_z'\circ\psi_z)'(x)| \right) \leq N,
\end{equs}
with $N=N(\sup|b|)$.
We define inductively the stopping times
\begin{equs}
\bar{\tau}_0:= 0, \qquad \bar{\tau}_{m+1}:= \inf \{ t \geq \tau_m : |X_t-X_{\tau_m } |>1\}, \qquad \tau_m := \bar{\tau}_m\wedge T .
\end{equs}
First we aim to obtain a bound for
\begin{equ}\label{eq:explanation0}
\sup_{t\in[0,T]}\E\big( \one_{\tau_m  \leq t} |X_{t\wedge\tau_{m+1}}-X^n_{t\wedge\tau_{m+1}}|^2\big).
\end{equ}
By definition of $\phi_z$,  we can find a function $g=g(z,x)$ such that for each $z \in \bR$,  $g(z,\cdot)$ is a version of $\phi_z''$, $\frac{1}{2}g(z,x)+b(x)\phi'(x)=0$ for all $x \in[z-2,z+2]$, and for each $x \in \bR$, $g(\cdot, x)$ is continuous in $z$.  We choose this function as a representative of $\phi_z''(x)$  for now on  and  by It\^o's formula (which, although $\phi''_z$ is not necessarily continuous, holds, see \cite[Thm.~22.5]{Kallenberg})  we have 
 we have  on $  \llbracket  \tau_m , \tau_{m+1} \rrbracket := \{ (\omega, t) \in \Omega \times [0,T]: t \in [\tau+m, \tau_{m+1}] \} $
\begin{equs}
\phi_z(X_t)= \phi_z(X_{\tau_m}) +\int_{\tau_m}^t \frac{1}{2}\phi''_z(X_s)+b(X_s) \phi'_z( X_s) \, ds + \int_{\tau_m}^t \phi'_z (X_s) \, dW_s.
\end{equs} 
Let $(\rho_\eps)_{\eps>0}$ be a molification family. We multiply the above equality with $\rho_\eps(z-X_{\tau_m})$, we integrate over $z \in \bR$ and we let  $\eps \to 0$, to obtain by virtue of the continuity of $\phi_z, \phi'_z$ and $\phi_z''$ in $z$ and the dominated convergence theorem 
 \begin{equs}
\phi_{X_{\tau_m }}(X_t)= \phi_{X_{\tau_m }}(X_{\tau_m}) + \int_{\tau_m}^t \frac{1}{2}\phi''_{X_{\tau_m }} (X_s)+b(X_s) \phi'_{X_{\tau_m }}( X_s) \, ds+\int_{\tau_m}^t \phi'_{X_{\tau_m }}(X_s) \, dW_s.
\end{equs} 
Using  \eqref{eq:Zvonkin PDE} (which holds for all $x \in [r-2,r+2]$) we have  for  $Y^m_t = \phi_{X_{\tau_m }}(X_t)$  on   $  \llbracket  \tau_m , \tau_{m+1} \rrbracket$ that 

\begin{equs}\label{eq:Y equation}
dY^m_t =  \phi_{X_{\tau_m }}(X_{\tau_m })+ \int_{\tau_m }^t \phi'_{X_{\tau_m }}\circ \psi_{X_{\tau_m }} (Y^m _s) \, dW_s.
\end{equs}
On $  \llbracket  \tau_m , \tau_{m+1} \rrbracket$ we define 
 \begin{equs}
 Y^{m,n}_t = \phi_{X_{\tau_m }}(X^n_{\tau_m })+ \int_{\tau_m}^t \phi'_{X_{\tau_m }}\circ \psi_{X_{\tau_m }}(Y^{m,n}_{ k_n(s)\vee \tau_m}) \, dW_s. 
 \end{equs}
 Then by \eqref{eq:all-Lip} and the triangle inequality we can bound the quantity in \eqref{eq:explanation0} as
 \begin{equs}[eq:explanation1]
 \E &\big(  \one_{\tau_m  \leq t}  |X_{t\wedge\tau_{m+1}}-X^n_{t\wedge\tau_{m+1}}|^2\big)  
 \leq   N \E \big(  \one_{\tau_m  \leq t} |Y^m_{t\wedge\tau_{m+1}}-\phi_{X_{\tau_m }}(X^n_{t\wedge\tau_{m+1}})|^2\big) 
\\
& \leq   N \E \big(  \one_{\tau_m  \leq t} |Y^m_{t\wedge\tau_{m+1}}-Y^{m,n}_{t\wedge\tau_{m+1}})|^2\big) +   N \E \big(  \one_{\tau_m  \leq t} |Y^{m,n}_{t\wedge\tau_{m+1}}-\phi_{X_{\tau_m }}(X^n_{t\wedge\tau_{m+1}})|^2\big),
\end{equs}
The first term on the right-hand side is simply the error of an Euler scheme for the SDE \eqref{eq:Y equation} with Lipschitz coefficient $\phi'_{X_{\tau_m }}\circ \psi_{X_{\tau_m }}$.
This is standard to estimate (for the convenience of the reader we include a short proof in the appendix):
 \begin{equs}   
 \E \big(  \one_{\tau_m  \leq t} |Y^m_{t\wedge\tau_{m+1}}-Y^{m,n}_{t\wedge\tau_{m+1}})|^2\big) &\leq    N \E |  \phi_{X_{\tau_m }}(X_{\tau_m })-  \phi_{X_{\tau_m }}(X^n_{\tau_m })|^2 + N n^{-1}
 \\
& \leq  N  \E |  X_{\tau_m }-  X^n_{\tau_m }|^2 +N n^{-1},   \label{eq:estimate-Y,Y^n}
 \end{equs}
 where we have used again \eqref{eq:all-Lip}  for the last inequality. 
 Moving on the second term on the right-hand side of \eqref{eq:explanation1},
by It\^o's formula again we have on $  \llbracket  \tau_m , \tau_{m+1} \rrbracket$
\begin{equs}
\phi_{X_{\tau_m}}(X^n_t)-Y^{m,n}_t&= \int_{\tau_m}^t   \phi_{X_{\tau_m}}'(X^n_s)b(X^n_{k_n(s)})+ \frac{1}{2} \phi_{X_{\tau_m}}''(X^n_s) \, ds 
\\
&\quad +   \int_{\tau_m}^t\big(  \phi_{X_{\tau_m}}'(X^n_s) - \phi_{X_{\tau_m}}' \circ \psi (Y^{m,n}_{k_n(s)\wedge \tau_m }) \big)  \, dW_s
\\
&= \int_{0}^t \one_{\tau_m\leq s}\one_{|X_s-X^n_s|<1}\left( \phi_{X_{\tau_m}}'(X^n_s)b(X^n_{k_n(s)}) - \phi_{X_{\tau_m}}'(X^n_s)b(X^n_s) \right) \, ds 
\\
&\quad + \int_{0}^t\one_{\tau_m\leq s}\one_{|X_s-X^n_s|\geq 1} \left(  \phi_{X_{\tau_m}}'(X^n_s)b(X^n_{k_n(s)})+ \frac{1}{2} \phi_{X_{\tau_m}}''(X^n_s) \right) \, ds 
\\
&\quad +   \int_{0}^t\one_{\tau_m\leq s}\big(  \phi_{X_{\tau_m}}'(X^n_s) - \phi_{X_{\tau_m}}' \circ \psi_{X_{\tau_m}} (Y^n_{k_n(s)\vee\tau_m}) \big)  \, dW_s
\\
&=: I^1_t+I^2_t+I^3_t.
\end{equs}
where for the second equality we have used \eqref{eq:Zvonkin PDE}. 
It immediately follows from \eqref{eq:all-Lip} that
\begin{equs}[11]
\E  \one_{\tau_m \leq t } |I^2_{t\wedge\tau_{m+1}}|^2 &\leq N  \E  \int_{0}^{t\wedge\tau_{m+1}}\one_{\tau_m\leq s}\one_{|X_s-X^n_s|\geq 1} \, ds=:J.
\end{equs}
Next, one can write
\begin{equs}[22]
 \E \one_{\tau_m \leq t } |I^1_{t\wedge\tau_{m+1}}|^2 
& \leq   N \E  \int_{\tau_m}^{t\wedge\tau_{m+1}} \one_{|X_s-X^n_s|\geq 1} \, ds 
\\
&\quad + 2 \E  \big|\int_{\tau_m}^{(\tau_m\vee t)\wedge\tau_{m+1}}
\big( \phi_{X_{\tau_m}}^\prime(X^n_s)b(X^n_{k_n(s)}) - \phi_{X_{\tau_m}}^\prime(X^n_s)b(X^n_s) \big) \, ds \big|^2 
\\
& \leq  J + N \int_0^T \E |X^n_s -X^n_{k_n(s)}|^2 \, ds 
\\
&\quad +   N  \E  \big|\int_{\tau_m}^{(\tau_m\vee t)\wedge\tau_{m+1}} 
\big( \phi_{X_{\tau_m}}^\prime(X^n_{k_n(s)})b(X^n_{k_n(s)}) - \phi_{X_{\tau_m}}^\prime(X^n_s)b(X^n_s) \big) \, ds \big|^2 
\\
&\leq J +  N n^{-1+\eps},
\end{equs}
where for the last term we have used Lemma \ref{lem:main-lemma-2}.
Concerning $I^3$, by It\^o's isometry and \eqref{eq:all-Lip} we have
\begin{equs}[33]
\E  \one_{\tau_m \leq t } |I^3_{t\wedge\tau_{m+1}}|^2 &  \leq N \E\int_0^{t\wedge\tau_{m+1}}   \one_{\tau_m \leq s  } | \phi_{X_{\tau_m}}(X^n_s)-Y^{m,n}_s   | \, ds 
\\
&\quad + N  \E  \int_{\tau_m }^{\tau_{m+1} } |Y^{m,n}_s-Y^{m,n}_{k_n(s)\wedge \tau_m} |^2 \, ds
\\
 &  \leq N \int_0^ t  \E \one_{\tau_m \leq s } | \phi_{X_{\tau_m}}(X^n_{s\wedge\tau_{m+1}})-Y^{m,n}_{s\wedge\tau_{m+1}}   | ^2 \, ds +N n^{-1}.
\end{equs} 
Putting \eqref{11}-\eqref{22}-\eqref{33} together, we arrive at
\begin{equs}
\E  \one_{\tau_m \leq t }| \phi_{X_{\tau_m}}(X^n_{t\wedge\tau_{m+1}})-Y^{m,n}_{t\wedge\tau_{m+1}}|^2 &  \leq N \int_0^ t  \E \one_{\tau_m \leq s } | \phi_{X_{\tau_m}}(X^n_{s\wedge\tau_{m+1}})-Y^{m,n}_{s\wedge\tau_{m+1}}|^2\,ds
 \\
 &\quad+N J +  N n^{-1+\eps}.
\end{equs}
Applying Gronwall's lemma to the function $t\mapsto\E  \one_{\tau_m \leq t }| \phi_{X_{\tau_m}}(X^n_{t\wedge\tau_{m+1}})-Y^{m,n}_{t\wedge\tau_{m+1}}|^2$, we get
\begin{equs} \label{eq:est-phi(x)-Y}
\E  \one_{\tau_m \leq t }| \phi_{X_{\tau_m}}(X^n_{t\wedge\tau_{m+1}})-Y^{m,n}_{t\wedge\tau_{m+1}}|^2\leq N J +  N n^{-1+\eps}.
\end{equs}
Putting \eqref{eq:explanation1}-\eqref{eq:estimate-Y,Y^n}-\eqref{eq:est-phi(x)-Y} together, we get
\begin{equ}\label{44}
\sup_{t\in[0,T]}\E \big(  \one_{\tau_m  \leq t}  |X_{t\wedge\tau_{m+1}}-X^n_{t\wedge\tau_{m+1}}|^2\big)  
\leq  N  \E |  X_{\tau_m }-  X^n_{\tau_m }|^2 +N n^{-1+\eps}+NJ.
\end{equ}
Notice that 
\begin{equ}
J\leq NT\sup_{t\in[0,T]}\E \big(  \one_{\tau_m  \leq t}  |X_{t\wedge\tau_{m+1}}-X^n_{t\wedge\tau_{m+1}}|^2\big).
\end{equ}
Therefore, provided that $T$ is sufficiently small, gives the last term on the right-hand side of \eqref{44} can be absorbed at a price of a constant factor. 
We then choose $T$ in place of $t$, which implies $\one_{\tau_m\leq T}=1$ and $T\wedge\tau_{m+1}=\tau_{m+1}$, yielding
\begin{equ}
\E|X_{\tau_{m+1}}-X^n_{\tau_{m+1}}|^2\leq N  \E |  X_{\tau_m }-  X^n_{\tau_m }|^2 +N n^{-1+\eps}.
\end{equ}
This by induction gives for $m \geq 1$
\begin{equs}
\E|X_{\tau_{m}}-X^n_{\tau_{m}}|^2 \leq N_0^m n^{-1+\eps},
\end{equs}
for some constant $N_0=N_0(C,\eps,\sup|b|)$ that we fix for the remainder of the proof. It remains to remove the stopping time from the estimate. We write
\begin{equs}
\E|X_T-X^n_T|^2
& \leq \E \big(  |X_{\tau_{m}}-X^n_{\tau_{m}}|^2\big) + \E \big(\one_{\tau_m< T} |X_T-X^n_T|^2\big)
\\
& \leq N_1^m n^{-1+\eps}+N \big(\bP( \bar{\tau}_m <   T)\big)^{1/2}.
\end{equs}
Define a new set of stopping times by
\begin{equs}
\hat{\tau}_0:= 0, \qquad \hat{\tau}_{m+1}:= \inf \{ t \geq \hat\tau_m : |W_t-W_{\tau_m } |>1/2\}.
\end{equs}
For $T< (2\sup|b|)^{-1}$, the contribution of the drift to $X$ is strictly smaller than $1/2$, and hence $\hat\tau_m\leq\bar\tau_m$.
Also, by the strong Markov property we have that 
the random variables $\{ \hat{\tau}_{m}-\hat{\tau}_{m-1} \}_{m=1}^\infty$ are independent and identically distributed.
Consequently, we have 
$$
\bP( \hat{\tau}_m \leq T )  \leq \prod_{i=1}^m \bP (\hat{\tau}_{i}-\hat{\tau}_{i-1} \leq T ) = \big( \bP (\hat{\tau}_{1} \leq T)\big)^m.
$$
As $T\to 0$, one has $\lim_{T \to 0} \bP(\hat{\tau}_1\leq T)=0$, and therefore $\alpha_T:=-\log\bP(\hat{\tau}_1\leq T)\to\infty$. Hence, we obtain for all $m \in \mathbb{N}$
\begin{equs}
\E|X_T-X^n_T|^2
&   \leq N_0^m n^{-1+\eps}+N e^{- \alpha_T m},
\end{equs}
which upon choosing $m = \ceil*{ \eps \frac{\ln n }{\ln N_0}}$ and assuming that $T$ is sufficiently small (so that  $\alpha_T$ is sufficiently large) gives
\begin{equs}
\E|X_T-X^n_T|^2
&   \leq N n^{-1+2\eps},
\end{equs} 
which brings the proof to an end. (Note that since $T$ was arbitrary in the interval $(0,T_0]$, the form \eqref{eq:main bound2} with the supremum outside of the expectation follows.)
\end{proof}

\begin{appendices}
\section[Appendix]{Appendix}
\begin{proof}[Proof of Lemma \ref{lem:PDE easy}.]
It is well known that the function 
\begin{equs}              \label{eq:solution representation}
u(t,x) = \int_0^t P_{t-s}f(s,x) \, ds,
\end{equs} 
where $(P_t)_{t \geq 0}$ is the heat semigroup, 
is the unique bounded distributional solution of equation \eqref{eq:dini-no-grad}.
One can write the action of the heat semigroup with its kernel in a convolution form $P_tg (x)=(p(t,\cdot)\ast g)(x)$ and recall  that for any $k=0,1,2$ and $t\in(0,1]$, one has 
\begin{equs}          \label{eq:blow-up-heat}
\|\partial_x^kp(t,\cdot)\|_{L_1(\R^d)}\leq N_0 t^{-k/2}
\end{equs}
It then follows immediately that one has the bounds
\begin{equs}    \label{eq:est-grad}
\|u\|_{L_\infty(Q_T)}+ \| \nabla u\|_{L_\infty(Q_T)} \leq N_0 \|f\|_{L_\infty(Q_T)}
\end{equs}
We now show that $\nabla^2 u \in C((0,T)\times \bR^d)$. 
First let us remark that from \eqref{eq:dini-cond} it follows by a simple change of variables that $\int_0^1\vartheta(r^\gamma)r^{-1}\,dr<\infty$ for any $\gamma\in(0,1)$. 
Notice that for any $t >0$, $g \in \mathcal{D}$, we have
\begin{equs}
|\D_{x_ix_j} P_{t}g (x)|
 =& \left|\int_{\bR^d} p_{x_ix_j}(t,x-y)(g(y)-g(x))\, dy\right| 
\\
\leq & N(d, \gamma) \left( \|g\|_\mathcal{D} \  \vartheta(t^{1/4}) t^{-1}+ \| g\|_{L_\infty} \left| \int_{|x| > t^{1/4}} p_{x_ix_j}(t, x) \, dx \right|\right) .
\end{equs} 
Since on $|x|>t^{1/4}$ we have $|x|^2/t>T^{-1/2}$, we see that 
\begin{align*}
\big| \int_{|x| > t^{1/4}}  p_{x_ix_j}(t, x) \, dx \big|
\leq & N(d) \int_{|x|>t^{1/4}} \frac{|x|^2}{t^{d/2+2} }\exp \left(\frac{-|x|^2}{2t }\right)\, dx 
\\
\leq & N(d,T)\int_{t^{1/4}}^ \infty r^{d-1}\frac{r^2}{t^{d/2+2}}\exp \left(\frac{-r^2}{2t }\right) \, dr 
\\
\leq & N(d,T) t^{-3/2}\int_{t^{1/4}}^\infty \exp\left(\frac{-r^2}{4t}\right)\,dr
\\
\leq & N(d,T) t^{-1} \exp\left(\frac{-t^{-1/2}}{4}\right)
\end{align*} 
Clearly, the right-hand side is integrable in $t$. We can conclude
\begin{equs}            \label{eq:est-second-derivative}
\| \D_{x_ix_j} P_t g \|_{L_\infty(\bR^d)} \leq N(d, \gamma) \lambda(t) \| g\|_{\mathcal{D}},
\end{equs}
where $\lambda \in L_1( 0,T)$. It follows then that for each $t \in (0,T)$, $u(t)$ is twice differentiable in space. Moreover, since 
\begin{equs}
\D_{x_i x_j} u(t,x)= \int_0^t \int_{\bR^d} \D_{x_i x_j}p(s,y) f(t-s, x-y) \, dy \, ds,
\end{equs}
the continuity of $\D_{x_i x_j}u$ in $(t,x)$ follows easily by the continuity of $f$ and \eqref{eq:est-second-derivative}, by virtue of Lebesgue's theorem on dominated convergence. The estimate 
\begin{equs}
\|\nabla^2 u\|_{L_\infty(Q_T)} \leq N_0 \sup_{t \leq T} \| f(t)\|_{\mathcal{D}}
\end{equs}
follows from \eqref{eq:est-second-derivative}. 
 The fact that $\D_tu \in C((0,T)\times \bR^d)$ now follows from the equation. Finally, \eqref{eq:dini-no-grad} follows from \eqref{eq:est-grad}, \eqref{eq:est-second-derivative}, and the equation. 
 
 We proceed with the proof for the remaining two estimates. Estimate \eqref{eq:est-u} follows immediately from \eqref{eq:solution representation} and  \eqref{eq:blow-up-heat}. Finally for the last estimate we have
 \begin{equs}      
 |\D_{x_i}u(t,x)-\D_{x_i}u(t,y)|&\leq   \|f\|_{L_\infty(Q_T)} \int_0^t\int_{\bR^d} |p_{x_i}(t-s,x-z)-p_{x_i}(t-s,y-z) |\, dz ds.
 \\
 \label{eq:Holder-derivative}
 \end{equs}
By \eqref{eq:blow-up-heat} we have  
 \begin{equs}
 \int_{\bR^d} |p_{x_i}(t-s,x-z)-p_{x_i}(t-s,y-z) |\, dz \leq N(d)  |t-s|^{-1/2},
 \end{equs}
 and by the fundamental theorem of calculus and \eqref{eq:blow-up-heat} we have 
 \begin{equs}
 \int_{\bR^d} |p_{x_i}(t-s,x-z)-p_{x_i}(t-s,y-z) |\, dz & \leq N(d) |x-y| \| \nabla p_{x_i}(t-s)\|_{L_1(\bR^d)}
 \\
&  \leq N(d) |x-y| |t-s|^{-1},
 \end{equs}
 which imply that for all $\alpha \in [0, 1]$ 
 \begin{equs}
 \int_{\bR^d} |p_{x_i}(t-s,x-z)-p_{x_i}(t-s,y-z) |\, dz \leq N(d)|x-y|^\alpha  |t-s|^{-(1+\alpha)/2}.
 \end{equs}
 This combined with $\eqref{eq:Holder-derivative}$ shows \eqref{eq:est-grad-Holder}. The lemma is proved.
 \end{proof}
 
 \begin{proof}[Proof of \eqref{eq:estimate-Y,Y^n}] 
 By definition of $\phi_z$ and \eqref{eq:all-Lip} we have a uniform (in $n$ and $m$) bound on the initial condition for $Y^{n,m}$:
 \begin{equs}
 \E  | \phi_{X_{\tau_m}}(X^n_{\tau_m})|^2 \leq N(1+ 
 \E |X^n_{\tau_m}|^2)\leq N(1+  
 \E \sup_{t \in [0,T]}  |X^n_t|^2)\leq N.
 \end{equs}
Combined with the (uniform in $z$) linear growth of the coefficient $\phi'_z \circ  \psi_z $, this implies by standard arguments
 \begin{equs}
 \E \sup_{t\in [\tau_m , \tau_{m+1}]} |Y^{m,n}_t|^2   \leq N. 
 \end{equs}
 From this we get 
 \begin{equs}
\E\sup_{t \in [\tau_m, \tau_{m+1}]} |Y^{m,n}_{k_n(t) \vee \tau_m}-Y^{m,n}_t|^2 & \leq \E \big| \int_{k_n(t) \vee \tau_m}^t \one_{s \leq \tau_{m+1}} \phi'_{X_{\tau_m}} \circ  \psi_{X_{\tau_m}}( Y^{m,n}_{k_n(s) \vee \tau_m}) \, ds \big|^2
\\
& \leq n^{-1} N(1+
\E \sup_{t\in [\tau_m , \tau_{m+1}]} |Y^{m,n}_t|^2 )  \leq N n^{-1}.
 \end{equs}
 By It\^o's formula again and the above inequality we have that 
 \begin{equs}
\  \E \one_{\tau_m \leq t}&|Y^{m,n}_{t\wedge {\tau_{m+1}}} -Y^m_{t \wedge\tau_{m+1} }|^2
  \\
    \leq & \E | \phi_{X_{\tau_m}}(X^n_{\tau_m}) - \phi_{X_{\tau_m}}(X_{\tau_m})|^2 
  \\
   &   \quad  +  \int_0^t  \E \one_{\tau_m \leq s  \leq \tau_{m+1} }|\phi'_{X_{\tau_m}} \circ   \psi_{X_{\tau_m}} (Y^{m,n}_{k_n(s) \vee \tau_m} )- \phi'_{X_{\tau_m}} \circ  \psi_{X_{\tau_m}} (Y^{m}_s )|^2 \, ds 
  \\
   \leq & \E | \phi_{X_{\tau_m}}(X^n_{\tau_m}) - \phi_{X_{\tau_m}}(X_{\tau_m})|^2 
  \\
 & \quad   + N \int_0^t   \E \one_{\tau_m \leq s \leq \tau_{m+1} } |Y^{m,n}_{k_n(s) \vee \tau_m}-Y^{m,n}_s|^2+  N \E \one_{\tau_m \leq s} |Y^{m,n}_{s \wedge \tau_{m+1}}-Y^m_{s \wedge \tau_{m+1}}|^2 \, ds
  \\
  \leq & \E | \phi_{X_{\tau_m}}(X^n_{\tau_m}) - \phi_{X_{\tau_m}}(X_{\tau_m})|^2 + N n^{-1}+ N \int_0^t  \E \one_{\tau_m \leq s} |Y^{m,n}_{s \wedge \tau_{m+1}}-Y^m_{s \wedge \tau_{m+1}}|^2 \, ds,
 \end{equs}
 and \eqref{eq:estimate-Y,Y^n} follows from Gronwall's lemma applied to the function $t\mapsto\E \one_{\tau_m \leq t}|Y^{m,n}_{t\wedge {\tau_{m+1}}} -Y^m_{t \wedge\tau_{m+1} }|^2$.

 \end{proof}

\end{appendices}


\bibliographystyle{Martin}

\end{document}